\documentclass[reqno]{amsart}

\usepackage[T1]{fontenc}
\usepackage{amsmath}
\usepackage{amsfonts}
\usepackage{verbatim}
\usepackage{bbm}
\usepackage{amsthm, amssymb}
\usepackage{authblk}

\newtheorem{theorem}{Theorem}
\newtheorem{lemma}[theorem]{Lemma}
\newtheorem{proposition}[theorem]{Proposition}
\newtheorem{corollary}[theorem]{Corollary}

\theoremstyle{definition}
\newtheorem{definition}[theorem]{Definition}

\theoremstyle{remark}
\newtheorem{remark}[theorem]{Remark}

\theoremstyle{remark}
\newtheorem{example}[theorem]{Example}

 \numberwithin{equation}{section}
\numberwithin{theorem}{section}
\title{Descent-Inversion Statistics in  Riffle Shuffles}
\author{{\"{U}mit I\c{s}lak}{USC}}

\date{March 10, 2013}
\begin{document}

%\affiliation[*]{University of Southern California}
 \maketitle

\begin{center}

\Small{\"{U}M\.{I}T I\c{S}LAK}

\end{center}
\begin{abstract}
This paper studies statistics of riffle shuffles by relating them to
 random word statistics with the use of inverse shuffles. Asymptotic normality of the number of descents and inversions in
riffle shuffles with convergence rates of order $1/\sqrt{n}$ in the
Kolmogorov distance are proven. Results are also given about the
lengths of the longest alternating subsequences of random
permutations resulting from riffle shuffles. A sketch of how the
theory of multisets can be useful  for statistics of a variation of
top $m$ to random shuffles is presented.
\end{abstract}

\section{Introduction}

For a sequence $\textbf{x}=(x_1,...,x_n)$ of real numbers,  the
number of descents and inversions are  defined as
$des(\textbf{x})=\sum_{i=1}^{n-1} \mathbbm{1}(x_i>x_{i+1})$ and
$inv(\textbf{x})=\sum_{i<j} \mathbbm{1}(x_i > x_j)$, respectively.
For a permutation $\pi$ in the symmetric group $S_n$, we write
$des(\pi)$ for the number of descents in the sequence
$(\pi(1),\pi(2),...,\pi(n)).$ Similar notation will be used for
inversions and other permutation statistics. In this paper, we will
analyze $des(\rho)$ and $inv(\rho)$ when $\rho$ is a random
permutation with riffle shuffle distribution (which is defined in
the next section precisely), and will discuss some other related
problems.

Our interest in descent-inversion statistics in riffle shuffles
started with the following elementary observation for uniformly
random permutations: Let $\pi$ be a uniformly random permutation in
$S_n$ and $\mathbb{X}=(X_1,...,X_n)$ be a random vector where
$X_i$'s are independent and identically distributed (i.i.d.)
$U(0,1)$ random variables. For $i=1,...,n$, let $R_i$ and $R_i'$ be
the ranks of $\pi(i)$ and $X_i$ in $(\pi(1),...,\pi(n))$ and
$(X_1,...,X_n)$ respectively. Then $(R_1,...,R_n)=_d
(R_1',...,R_n')$ where $=_d$ denotes equality in distribution.

This simple result, which can be proven by a simple induction (or,
by a measure theoretic argument as in \cite{houdre}), makes it
easier to study problems regarding uniform permutation statistics by
transforming them into independent $U(0,1)$ random variable
statistics. As an example, we have $inv(\pi) =_d \sum_{i< j}
\mathbbm{1}(X_i
> X_j)$ giving an alternative representation of
$inv(\pi)$ that can be quite useful for asymptotic problems.

A natural question at this point is: What would $\sum_{i< j}
\mathbbm{1}(X_i > X_j)$ represent if $X_i'$s were instead i.i.d.
over $[a]:=\{1,...,a\}$ with distribution $\textbf{p}=(p_1,...,p_a)$
where $a \geq 2$? Recently, Bliem and Kousidis \cite{kousidis} and
Janson \cite{Janson} considered this problem in terms of the
generalized Galois numbers and provided several different
probabilistic explanations.

In this paper, we give a different interpretation of this using
random permutations which is analogous to the discussion given above
for uniformly random permutations. This time, the equivalent
distribution turns out to be a biased riffle shuffle with $a$ hands.
Using this transformation, we are able to obtain asymptotic
normality of the number of inversions in riffle shuffles (which was
questioned in \cite{fulman}, pg 10) with convergence rates,  and
also understand some other related statistics.

The organization of this paper is as follows. Section 2 provides
background in riffle shuffles and makes the connection to random
words using inverse shuffles. It also discusses how similar results
can be obtained for a variation of top $m$ to random shuffles.
Section 3 treats  the asymptotic distribution of the number of
descents and inversions in riffle shuffles. Section 4 provides
asymptotic results for the lengths of longest alternating
subsequences in uniformly random permutations and riffle shuffles.

\section{Riffle shuffles and connection to random
words}\label{riffle}

The method most often used to shuffle a deck of cards is the
following: first,  cut the deck into two piles and then riffle the
piles together, that is, drop the cards from the bottom of each pile
to form a new pile. The first mathematical models for riffle
shuffles were introduced \cite{gilbert} and \cite{reeds}. These were
further developed  in \cite{diac} and \cite{fulman}. Now following
\cite{fulman}, we will give two equivalent descriptions of riffle
shuffles in the most general sense. For other alternative
descriptions (which will not be used in this paper), see
\cite{diac2} and \cite{fulman}.

\vspace{0.1in}

\noindent \textbf{ Description 1 : } Cut the $n$ card deck into $a$
piles by picking pile sizes
  according to the $mult(a;\textbf{p})$ distribution, where $\textbf{p}=(p_1,...,p_a)$.
That is, choose $b_1,...,b_a$ with probability
$$\binom{n}{b_1,...,b_a} \Pi_{i=1}^a p_i^{b_i} .$$ Then choose
uniformly one of the $\binom{n}{b_1,...,b_a}$ ways of interleaving
the packets, leaving the cards in each pile in their original order.
\vspace{0.1in}
\begin{definition}
The probability distribution on $S_n$ resulting from Description 1
will be called as the \emph{riffle shuffle distribution} and will be
denoted by $P_{n,a,\textbf{p}}.$  When
$\textbf{p}=(1/a,1/a,...,1/a)$, the shuffle is said to be
\emph{unbiased} and the resulting probability measure is denoted by
$P_{n,a}.$ Otherwise, shuffle is said to be \emph{biased.}
\end{definition}

 Note that the usual way of shuffling $n$ cards corresponds to $P_{n,2}$ (assuming
that the shuffler is not cheating). Before moving on to Description
2, let's give an example using unbiased 2-shuffles. The permutation
$$\rho_{n,2}=\left(
                  \begin{array}{ccccccc}
                    1 & 2 & 3 & 4 & 5 & 6 & 7 \\
                    1 & 2 & 5 & 3 & 6 & 7 & 4 \\
                  \end{array}
                \right)$$
is a possible outcome of the $P_{n,2}$ distribution. Here the first
four cards form the first pile, the last three form the second one
and these two piles are riffled together.

The following alternative description will be  important in the
following discussion.

\vspace{0.1in}

\noindent \textbf{Description 2 :} (Inverse $a$-shuffles) The
inverse of a biased a-shuffle has the following description. Assign
independent random digits from
  $\{1,...,a\}$ to each card with distribution $\textbf{p}=(p_1,...,p_a)$. Then sort according to digit,
  preserving relative order for cards with the same digit.

\vspace{0.1in}

In other words, if $\sigma$ is generated according to Description 2,
then $\sigma^{-1} \sim P_{n,a,\textbf{p}}.$ A proof of the
equivalence of these two descriptions (with two other formulations)
for unbiased shuffles can be found in \cite{diac}. Extension to
biased case is straightforward. Now let's  give an example of
generating a random permutation with distribution $P_{n,2}$ using
inverse shuffles.

Consider a deck of 7 cards. We wish to shuffle this deck with the
unbiased 2-shuffle distribution using inverse shuffles. Let
$\mathbb{X}=(X_1,...,X_n)=(1,1,2,1,2,2,1)$ be a sample from
$U(\{1,2\}^7)$. Then, sorting according to digits   preserving
relative order for cards with the same digit gives the new
configuration of cards as $(1,2,4,7,3,5,6)$. In the usual
permutation notation, the resulting permutation after the inverse
shuffle is
$$\sigma= \left(
                                                \begin{array}{ccccccc}
                                                  1 & 2 & 3 & 4 & 5 & 6 & 7 \\
                                                  1 & 2 & 4 & 7 & 3 & 5 & 6 \\
                                                \end{array}
                                              \right),
$$
and the resulting sample from $P_{n,2}$ is
$$\rho_{n,2}:=\sigma^{-1}= \left(
                  \begin{array}{ccccccc}
                    1 & 2 & 3 & 4 & 5 & 6 & 7 \\
                    1 & 2 & 5 & 3 & 6 & 7 & 4 \\
                  \end{array}
                \right).
$$

In the following, we will sometimes call the random vector
$\mathbb{X}=(X_1,...,X_n)$ where $X_i$'s are independent with
distribution $\textbf{p}=(p_1,...,p_a)$  as a \emph{random word}.

Next we formalize the relation between riffle shuffles and random
words. Let  $\rho_{n,a,\textbf{p}}$ be a random permutation with
distribution $P_{n,a,\textbf{p}}$ that is generated using inverse
shuffles with the random word $\mathbb{X}=(X_1,...,X_n)$ and
 observe that
$$\rho_{n,a,\textbf{p}}(i)= \#\{j:X_j < X_i\}+\#\{j\leq i : X_j = X_i\}.$$

Thus for $i,k \in [n]$, we have
$\rho_{n,a,\textbf{p}}(i)>\rho_{n,a,\textbf{p}}(k)$ if and only if
$$\#\{j:X_j < X_i\}+\#\{j : j\leq i , X_j = X_i\}>\#\{j:X_j <
X_{k}\}+\#\{j : j\leq k, X_j = X_{k}\}.$$ Using this, for the case
$i<k$, we immediately arrive at the following important lemma.

\begin{lemma}\label{import} Let $\mathbb{X}=(X_1,...,X_n)$ where $X_i$'s are independent with distribution $\textbf{p}=(p_1,...,p_a)$.
Also let $\rho_{n,a,\textbf{p}}$ be the corresponding permutation as
described above so that $\rho_{n,a,\textbf{p}}$ has distribution
$P_{n,a,\textbf{p}}.$ Then for $i<k$, $\rho_{n,a,\textbf{p}}(i)>
\rho_{n,a,\textbf{p}}(k)$ if and only if $X_i > X_{k}.$
\end{lemma}

This has the following corollary:

\begin{corollary}\label{inverse}
Consider the setting in Lemma \ref{import} and let $S \subset
\{(i,j) \in [n]\times [n]: i < j \}.$ Then $$\sum_{(i,j) \in S}
\mathbbm{1}(\rho_{n,a,\textbf{p}}(i)> \rho_{n,a,\textbf{p}}(j)) =
\sum_{(i,j) \in S} \mathbbm{1} (X_i > X_j).$$
\end{corollary}

In the following two sections,  we will make use of this connection
to study various statistics of riffle shuffles. Before that, we
demonstrate the use of random words approach with two other
examples. The first one will be relating riffle shuffles to
uniformly random permutations and the second one will give a
different interpretation of a variation of top to random shuffles.
As a general remark, we note that the results in this paper will be
mostly given for unbiased shuffles to keep the notations simple.
However, all the results in this paper are extendible to the biased
case in a straightforward way.

We start with a total variation result relating riffle shuffle
statistics and uniform permutation statistics. Although the result
is given for $des$ and $inv$, it is much more general as can be seen
from the proof easily.

\begin{theorem}\label{compareunif}
Let $\rho_{n,a}$ and $\pi$ be random permutations in $S_n$ with
unbiased $a-$shuffle distribution and uniform distribution,
respectively. If $f=des$ or $f=inv$, then for any $a \geq n$,
$$d_{TV} (f(\rho_{n,a}), f(\pi)) \leq 1- \frac{a!}{(a-n)!}
\frac{1}{a^n}.$$ In particular, $d_{TV}(f(\rho_{n,a}), f(\pi))
\rightarrow 0$ as $a \rightarrow \infty$.
\end{theorem}

\begin{proof}
Let $\pi$ be a uniformly random permutation and $\rho_{n,a}$ be a
random permutation with distribution $P_{n,a}$ that is generated
using inverse shuffling with the random vector
$\mathbb{X}=(X_1,...,X_n)$. Also let $T$ be the number of different
digits in the vector $\mathbb{X}.$ Then we have
\begin{eqnarray}\label{unifriff}
  \mathbb{P}(f(\rho_{n,a})\in A ) &=& \mathbb{P}(f(\rho_{n,a})\in A , T=n) +  \mathbb{P}(f(\rho_{n,a})\in A, T<n ) \nonumber\\
    &=& \mathbb{P}(f(\rho_{n,a})\in A|T=n) \mathbb{P}(T=n) +\mathbb{P}(f(\rho_{n,a})\in
    A, T<n) \nonumber\\
    &\leq& \mathbb{P}(f(\pi)\in A) +\mathbb{P}(f(\rho_{n,a})\in A,
    T<n)
\end{eqnarray}
where (\ref{unifriff}) follows by observing
$\mathbb{P}(f(\rho_{n,a})\in A|T=n)=\mathbb{P}(\pi \in A)$ since
$\rho_{n,a}$ has uniform distribution conditional on $T=n$. This
yields
\begin{eqnarray}\label{totvar1}
  \mathbb{P}(f(\rho_{n,a}) \in A) -    \mathbb{P}(f(\pi) \in A) &\leq&   \mathbb{P}(f(\rho_{n,a}) \in A,
  T<n) \leq \mathbb{P}(T<n).
\end{eqnarray}
Similarly, we have
\begin{eqnarray*}
  \mathbb{P}(f(\pi) \in A) &=& \mathbb{P}(f(\pi) \in A)\mathbb{P}(T=n) + \mathbb{P}(f(\pi) \in A)\mathbb{P}(T<n) \\
    &\leq& \mathbb{P}(\rho_{n,a} \in A) + \mathbb{P}(T<n)
\end{eqnarray*}
implying
\begin{equation}\label{totvar2}
\mathbb{P}(f(\pi) \in A) -\mathbb{P}(f(\rho_{n,a}) \in A) \leq
\mathbb{P}(T>n).
\end{equation}
Hence combining (\ref{totvar1}) and (\ref{totvar2}), for $a \geq n$,
we get
\begin{eqnarray*}
  d_{TV}(f(\rho_{n,a}), f(\pi))\leq \mathbb{P}(T<n) =
 \mathbb{P} \left( \bigcup_{i \neq j} \{X_i = X_j\} \right)
   &=& 1 - \mathbb{P}
\left(\bigcap_{i \neq j} \{X_i \neq X_j\} \right) \\
&=& 1 - \frac{\binom{a}{n} n!}{a^n}  \\
&=& 1- \frac{a!}{(a-n)!} \frac{1}{a^n}
\end{eqnarray*}
proving the first claim. The second assertion is immediate from the
bound we obtained.
\end{proof}

\begin{remark}
As can be seen easily from the proof, the result is actually true
for a large class of functions $f.$
\end{remark}

\begin{remark}\label{convolution}
Note that Theorem \ref{compareunif} is also informative for
understanding multiple 2-shuffles by a nice convolution property of
riffle shuffles given by Fulman \cite{fulman}. Letting
$\textbf{p}=(p_1,...,p_a)$, $\textbf{p}'= (p_1',...,p_b')$ be two
probability measures and defining the product $\otimes$ by
$\textbf{p} \otimes \textbf{p}' = (p_1 p_1',...,p_1 p_b',...,p_a
p_1',...,p_a p_b'),$ Fulman's result gives that the convolution of
$P_{n,a,p}$ and $P_{n,a,p'}$ is $P_{n,ab,p \otimes p'}.$ In
particular, when $a=b=2$ and $p_1=p_2=1/2,$ the case of multiple
2-shuffles is handled.
\end{remark}

Since convergence in total variation implies convergence in
distribution, we also have

\begin{corollary} If the shuffle is unbiased, then $f(\rho_{n,a})
\longrightarrow_d  f(\pi)$ as $a \rightarrow \infty.$
\end{corollary}

We close this section by describing how one can use above ideas to
study a variation of top $m$ to random shuffles which was first
introduced in \cite{diac3}. Consider a deck of $n$ cards and let $0
\leq m \leq n$ be fixed. Now cut off the top $m$ cards and insert
them randomly among the remaining $n-m$ cards, keeping both packets
in the same relative order. We will call this shuffling method as
\emph{ordered top $m$ to random shuffles}.

An ordered top $m$ to random shuffle is actually equivalent to a
2-shuffle in which exactly $m$ cards are cut off (whereas for the
2-shuffles case, $m$ is a binomial random variable). It is not hard
to see that the following result gives an inverse description of
ordered top $m$ to random shuffles.

\begin{theorem}\label{inversefortop}
The inverse of an ordered top $m$ to random shuffle has the
following description. Assign card $i\in [n]$ a random bit $X_i$
where the random vector $\mathbb{X}=(X_1,...,X_n)$ is uniformly
distributed over $\{0,1\}^n$ with the restriction that $\sum_{i=1}^n
X_i = n-m.$ Then sort according to digit, preserving relative order
for cards with the same digit.
\end{theorem}

Now letting $\tau$ be a random permutation in $S_n$ with ordered top
$m$ to random shuffle distribution, Theorem \ref{inversefortop}
allows us to rewrite $des(\tau)$ or $inv(\tau)$ in a useful way
exactly as we did in Corollary \ref{inverse}. Namely,  we have
$$des(\tau) =_d \sum_{i=1}^{n-1} \mathbbm{1}(X_i> X_{i+1}) \qquad \text{and} \qquad inv(\tau)=_d \sum_{i<j}
\mathbbm{1}(X_i>X_j)$$ where $\mathbb{X}=(X_1,...,X_n)$ is uniformly
distributed over $\{0,1\}^n$ with the restriction that $\sum_{i=1}^n
X_i = n-m.$ Hence the problem is transformed into a problem of
uniform permutations of a fixed multiset which is well studied in
the literature. See, for example, \cite{Vis}. We will revisit this
at the end of Section \ref{descinv}.

\section{Convergence rates for the number of descents and
inversions}\label{descinv}

In this section we will discuss the asymptotic normality of the
number of descents and inversions in  riffle shuffles and will
provide convergence rates of order $1/\sqrt{n}$ in the Kolmogorov
distance. Recall that the Kolmogorov distance between two
probability measures $\mu$ and $\nu $ on $\mathbb{R}$ is defined to
be
$$d_K(\mu, \nu) = \sup_{z \in \mathbb{R}}|\mu((-\infty,z]) -
\nu((-\infty,z]) |.$$ We start with the asymptotic normality of the
number of inversions after an $a$ shuffle which was conjectured by
Fulman in \cite{fulman} for unbiased 2-shuffles. Our strategy will
be using Corollary \ref{inverse} to transform the problem  into
random words language, use Janson's U-statistic construction
\cite{Janson} for the random words case  and finally use Chen and
Shao's results on asymptotics of U-statistics \cite{chen}. Before
moving on to the main result, we provide some pointers to the
literature and give the necessary background on U-statistics.

First we note that the asymptotic normality of the number of
inversions in random words is recently proven by Bliem and Kousidis
\cite{kousidis} without convergence rates in a more general
framework.  In \cite{Janson}, Janson gave equivalent descriptions of
the random words problem and analyzed the asymptotic behavior using
U-statistics theory. Naturally, the convergence rate result given
here will also apply to Janson's case.

Now recall that for a real valued symmetric function $h :
\mathbb{R}^m \rightarrow \mathbb{R}$ and for a random sample
$X_1,...,X_n$ with $n \geq m$, \emph{a U-statistic with kernel} $h$
is defined as
$$U_n=U_n(h) = \frac{1}{\binom{n}{m}} \sum_{C_{m,n}}
h(X_{i_1},...,X_{i_m})$$ where the summation is over the set
$C_{m,n}$ of all $\binom{n}{m}$ combinations of $m$ integers, $i_1 <
i_2<...<i_m$ chosen from $\{1,...,n\}.$ The next result of Chen and
Shao will be useful for obtaining convergence rates in the
Kolmogorov distance. We note that throughout this paper, $Z$ will
denote a standard normal random variable. Also in the following
statement $h_1(X_1):=\mathbb{E}[h(X_1,...,X_m)|X_1]$.
\begin{theorem}\label{conv} \cite{chen}
Let $X_1,...,X_n$ be i.i.d. random variables, $U_n$ be a U-statistic
with symmetric kernel $h$, $\mathbb{E}[h(X_1,...,X_m)]=0,
\sigma^2=Var(h(X_1,...,X_m))<\infty$ and
$\sigma_1^2=Var(h_1(X_1))>0.$  If in addition
$\mathbb{E}|h_1(X_1)|^3< \infty,$ then $$d_K \left(\frac{\sqrt{n}}{m
\sigma_1}U_n,Z \right)  \leq \frac{6.1
\mathbb{E}|h_1(X_1)|^3}{\sqrt{n} \sigma_1^3}
+\frac{(1+\sqrt{2})(m-1)\sigma}{(m(n-m+1))^{1/2} \sigma_1}.
$$
\end{theorem}

Now we are ready to state and prove our main result on the number of
inversions in riffle shuffles.

\begin{theorem}\label{normalinversion}
Let $\rho_{n,a}$ be a random permutation with distribution $P_{n,a}$
with $a\geq 2.$ Then
$$d_K\left(\frac{inv(\rho_{n,a})-\frac{n(n-1)}{4}
\frac{a-1}{a}}{\sqrt{n}(n-1) \sqrt{\frac{a^2-1}{36a^2}}}, Z\right)
\leq \frac{C}{\sqrt{n}}$$ where $C$ is a constant independent of
$n$.
\end{theorem}

\begin{proof}

Let $a \geq 2$ and  $\rho_{n,a}$ have distribution $P_{n,a}$. Using
Corollary \ref{inverse}, we have
$$inv(\rho_{n,a}) := \sum_{i<j} \mathbbm{1}(\rho_{n,a}(i)> \rho_{n,a}(j))=_d\sum_{i<j}
\mathbbm{1}(X(i)>X(j))$$ where $X_i'$s are independent and uniformly
distributed over $[a]$. This immediately yields
\begin{eqnarray*}
  \mathbb{E}[inv(\rho_{n,a})] &=& \mathbb{E} \left[ \sum_{i<j} \mathbbm{1}(X_i>X_j)\right]
    = \binom{n}{2} \mathbb{P}(X_1 > X_2) = \frac{n(n-1)}{4} \frac{a-1}{a}.
\end{eqnarray*}

Using similar elementary computations one gets

$$\sigma^2=Var(inv(\rho_{n,a}))=\frac{n(n-1)(2n+5)}{72} \frac{a^2-1}{a^2}.$$ See \cite{kousidis} or \cite{Janson} for details.
Now following \cite{Janson}, we will find a U-statistic
representation of $inv(\rho_{n,a}).$ All details are included for
the sake of completeness.

Let $U_1,...,U_n$ be independent random variables uniformly
distributed over $(0,1).$ Order $U_i$'s as
$U_{\sigma(1)}<U_{\sigma(2)}<...<U_{\sigma(n)}$ where $\sigma \in
S_n$ is properly chosen. Since $\sigma$ has uniform distribution
over $S_n$,  we have

\begin{equation}\label{equallaw}
(X_1,...,X_n)=_d (X_{\sigma(1)},...,X_{\sigma(n)}).
\end{equation}

 Now we get
\begin{eqnarray*}
                            inv(\rho_{n,a}) =_d \sum_{i<j} \mathbbm{1}(X_i>X_j)
                              &=_d& \sum_{i<j}
                              \mathbbm{1}(X_{\sigma(i)}>X_{\sigma(j)})
                              \\
                              &=& \sum_{i,j=1}^n \mathbbm{1}(X_{\sigma(i)}>X_{\sigma(j)}, i
                              <j).
                          \end{eqnarray*}

\noindent where the second equality follows from (\ref{equallaw}).
Observing $i<j$ if and only if $U_{\sigma(i)}< U_{\sigma(j)}$, we
obtain
\begin{eqnarray}\label{inveqn1}
                                                                       inv(\rho_{n,a}) =_d  \sum_{i,j=1}^n \mathbbm{1}(X_{\sigma(i)}>X_{\sigma(j)}, U_{\sigma(i)}<U_{\sigma(j)})
                                                                       = \sum_{i,j=1}^n \mathbbm{1}(X_i>X_j) \mathbbm{1}(U_i < U_j)
                                                                    \end{eqnarray}
Next let $Z_i = (X_i,U_i)$ for $i=1,...,n$ and observe that $Z_i$'s
are i.i.d. random variables. Define the functions $f$ and $g$ by
$$f((x_i,u_i),(x_j,u_j)):= \binom{n}{2}\mathbbm{1}(x_i>x_j) \mathbbm{1} (
u_i<u_j)$$ and $$g((x_i,u_i),(x_j,u_j))=
f((x_i,u_i),(x_j,u_j))+f((x_j,u_j),(x_i,u_i)).$$ Then clearly $g$ is
a real valued symmetric function and

\begin{equation}\label{inveqn2}
    \sum_{k,l=1}^n \mathbbm{1}(X_k>X_l) \mathbbm{1}(U_k < U_l) =
\frac{1}{\binom{n}{2}}\sum_{k < l} g(Z_k,Z_l).
\end{equation}
 Thus, by (\ref{inveqn1}) and (\ref{inveqn2}) we conclude
that $inv(\rho_{n,a})$ is a U-statistic with $$inv(\rho_{n,a})=_d
\binom{n}{2}^{-1} \sum_{i<j} \binom{n}{2} (\mathbbm{1}(X_i
>X_j)\mathbbm{1}(U_i <U_j)+\mathbbm{1}(X_i <X_j)\mathbbm{1}(U_i
>U_j)).$$ So in terms of Theorem \ref{conv}, we have
$h((x_1,u_1),(x_2,u_2))=\binom{n}{2}k((x_1,u_1),(x_2,u_2))$ where
$$k((x_1,u_1),(x_2,u_2))=\mathbbm{1}(x_1 >x_2)\mathbbm{1}(u_1
<u_2)+\mathbbm{1}(x_1 <x_2)\mathbbm{1}(u_1 >u_2)-\frac{a-1}{2a}.$$
Defining
$$k_1(x_1,u_1)=\mathbb{E}[k(X_1,U_1),(X_2,U_2)|X_1=x_1,U_1=u_1],$$
we have $h_1(x_1,u_1)=\binom{n}{2}k_1(x_1,u_1).$ Also

\begin{eqnarray*}
  k_1(X_1,U_1) &=& \mathbb{E}[\mathbbm{1}(X_1 >X_2)\mathbbm{1}(U_1 <U_2)+\mathbbm{1}(X_1 <X_2)\mathbbm{1}(U_1 >U_2)|X_1,U_1]-\frac{a-1}{2a} \\
  &=& \frac{X_1-1}{a}(1-U_1)+\frac{a-X_1}{a}U_1 -\frac{a-1}{2a} \\
    &=&  \frac{1}{a}(X_1-2X_1U_1+(a+1)U_1-1)-\frac{a-1}{2a}.
\end{eqnarray*}

Now doing some elementary computations, we obtain
$$\sigma_1^2=Var(h_1(X_1,U_1))=\binom{n}{2}^2 Var(k_1(X_1,U_1))=\binom{n}{2}^2
\frac{a^2-1}{36a^2}$$ and also $$\mathbb{E}|h_1(X_1,U_1)|^3 \leq 9
\binom{n}{2}^3.$$ Hence using Theorem \ref{conv}, we arrive at
$$d_K\left(\frac{inv(\rho_{n,a})-\frac{n(n-1)}{4}
\frac{a-1}{a}}{\sqrt{n}(n-1) \sqrt{\frac{a^2-1}{36a^2}}}, Z\right)
\leq \frac{(6.1) 9 \binom{n}{2}^3}{\sqrt{n}\binom{n}{2}^3
\left(\frac{a^2-1}{36a^2} \right)^{3/2}}
+\frac{(1+\sqrt{2})\sqrt{\frac{n(n-1)(2n+5)}{72}}
\sqrt{\frac{a^2-1}{a^2}}}{\sqrt{2}\sqrt{n-1}
\binom{n}{2}\sqrt{\frac{a^2-1}{36a^2}}}$$ which in particular
implies the existence of a constant $C$ independent of $n$ as in the
statement of the theorem. This completes the proof.
\end{proof}
\begin{remark}
U-statistics construction given above will still work when the
shuffle is biased.  So under certain conditions on the distribution
vector $\textbf{p}=(p_1,...,p_a)$ (namely, by excluding the case
$p_j=1$ for some $j \in [a]$), one can extend Theorem
\ref{normalinversion} to the case of biased riffle shuffles (or
random words).
\end{remark}

\begin{remark}
By the nice convolution property discussed in Remark
\ref{convolution}, Theorem \ref{normalinversion} also gives
convergence rates for multiple unbiased $2-$shuffles (with explicit
constants, as can be seen easily from the proof).
\end{remark}

Next we move on to the number of  descents in riffle shuffles which
is much easier due to the underlying local dependence. Recall that,
if we define the distance between two subsets of $A$ and $B$ of
$\mathbb{N}$ by $$\rho(A,B):= \inf \{|i-j|: i \in A, j \in B\},$$
the sequence of random variables $Y_1,Y_2,...$ is said to be
\emph{$m-$dependent} if $\{Y_i, i \in A\}$ and $\{Y_j, j \in B\}$
are independent whenever $\rho(A,B) > m$ with $A, B \subset
\mathbb{N}.$ Now we recall the following result from
\cite{chenshaolocal} about $m-$dependent random variables.

\begin{theorem}\label{localdepresult} \cite{chenshaolocal} If $\{Y_i\}_{i \geq 1}$ is  a sequence of zero mean
$m-$dependent random variables, $W=\sum_{i=1}^{n} Y_i$ and
$\mathbb{E}[W^2]=1$, then for all $p \in (2,3]$, $$d_K(W,Z) \leq 75
(10m+1)^{p-1} \sum_{i=1}^{n} \mathbb{E}|Y_i|^p.$$
\end{theorem}

Now, letting $\rho_{n,a}$ be a sample from $P_{n,a}$, we know from
Corollary \ref{inverse} that
$$des(\rho_{n,a}) = \sum_{i=1}^{n-1} \mathbbm{1}(\rho_{n,a}(i)>\rho_{n,a}(i+1))=_d \sum_{i=1}^{n-1}
\mathbbm{1}(X_i>X_{i+1})$$ where $X_i's $ are independent and
uniform over $[a].$ Setting $V= \sum_{i=1}^{n-1} Y_i$ with $Y_i =
\mathbbm{1}(X_i>X_{i+1})$, we have $\mathbb{E}[V] = (n-1)
\frac{a-1}{2a}.$ Also since $Var(Y_i)= \frac{a^2-1}{4a^2}$ and
$Cov(Y_i,Y_{i+1})=-\left(\frac{a^2-1}{12a^2}\right)$ for
$i=1,...,n-1$, we have

\begin{eqnarray*}
  Var(V)=\sum_{i=1}^{n-1} Var(Y_i) +2 \sum_{i<j}
Cov(Y_i,Y_j) &=& (n-1)\frac{a^2-1}{4a^2}-2(n-1)\left(\frac{a^2-1}{12a^2}\right) \\
    &=& \frac{(a^2-1)(n-1)}{12a^2}.
\end{eqnarray*}
Now noting that $Y_i$'s are 1-dependent and using Theorem
\ref{localdepresult} with $p=3$, we arrive at

\begin{theorem}\label{descents}
Let $\rho_{n,a}$ be distributed according to $P_{n,a}.$ Then
$$d_K\left(\frac{des(\rho_{n,a})-\frac{(a-1)(n-1)}{2a}
}{\sqrt{\frac{(a^2-1)(n-1)}{12a^2}}}, Z\right) \leq
\frac{C}{\sqrt{n}}$$ where $C$ is a constant independent of $n$.
\end{theorem}

\begin{remark}\label{stocdom}
The discussion from Section \ref{riffle}, and a simple coupling
argument gives the following stochastic dominance result, say, for
the number of inversions:
$$Inv(\rho_{n,2}) \leq_s  Inv(\rho_{n,a})\leq_s Inv(\pi) $$
where $a \geq 2$, $\pi$ is a uniformly random permutation and
$\leq_s$ denotes stochastic ordering. Since the means and variances
of these three statistics are of the same order, it wouldn't be
surprising to obtain the asymptotic normality of $Inv(\rho_{n,a})$
by the corresponding results for $Inv(\rho_{n,2})$ and $Inv(\pi).$
We will pursue this idea in a future work.
\end{remark}

We conclude this section with a discussion of the asymptotic
normality of the number of inversions after ordered top $m$ to
random shuffles which were defined at the end of Section
\ref{riffle}. We start by recalling a special case of a result of
Congar and Viswanath \cite{Vis} on multisets. Let $\beta \in
[1/2,1)$. Then there exists a constant $C
> 0 $ depending only on $\beta$ so that whenever $\tau$ is a uniform permutation of the
multiset $\{0^{n_0}, 1^{n_1}\}$ with $n_0, n_1 \in \mathbb{N}$,
$n_0+n_1=n$, $\max\{n_0,n_1\} \leq \beta n$,
$$d_K\left(\frac{des(\tau)-
\mu}{\sigma},Z \right) \leq \frac{C}{\sqrt{n}}$$ is satisfied where
$\mu=\mathbb{E}[des(\tau)]$ and $\sigma^2=Var(des(\tau))$ (For
details, see \cite{Vis}). It is easily seen from this result and
Theorem \ref{inversefortop} that, one can analyze the asymptotic
behavior of the number of inversions in ordered top $m$ to random
shuffles under the assumption that $\max\{m,n-m\} \leq \beta n.$
Note that this also suggests a natural generalization of riffle
shuffles. To see this, consider the case where the number of cards
in the hands are $(n_0,n_1)$ where $(n_0,n_1)$ is uniform over the
set $\{(n_0,n_1) \in [n] \times [n] : n_0+n_1=n, \min
\{n_0,n_1\}\geq \alpha n \}$ for some $1 > \alpha \geq 0.$ When
$\alpha= 0$, we get $P_{n,2}.$ Using $\alpha>0$, we get a different
model which can be meaningful since when one shuffles a deck, there
will be at least a few cards in each hand.

\section{Another related Statistic : Longest alternating subsequences}

In this section we will study the asymptotic behavior of lengths of
longest alternating subsequences (which are closely related to
descents) in uniform permutations and riffle shuffles. Letting
$\textbf{x}:=(x_i)_{i=1}^n$ be a sequence of real numbers, a
subsequence $x_{i_k}$, where $1 \leq i_1 < ...< i_k \leq n $, is
called an \emph{alternating subsequence} if
$x_{i_1}>x_{i_2}<x_{i_3}>...x_{i_k}.$ The \emph{length of the
longest alternating subsequence of } $\textbf{x}$ is defined as
$$LA_n(\textbf{x}):=\max \{k: \textbf{x} \; \text{has an alternating subsequence of
length } k\}.$$ For an example, let $\textbf{x}=(3,1,7,4,2,6,5).$
Then $(3,1,7,2,6,5)$ is an alternating subsequence and it is easy to
see that $LA_7(\textbf{x})=6.$ For an excellent survey on longest
alternating subsequence problem, see \cite{stanley}.  The following
lemma, whose proof can be found in \cite{houdre} and \cite{romik},
is very useful to understand $LA_n(\textbf{x})$ when $\textbf{x}$ is
a sequence of random variables.

\begin{lemma}\label{lakey} \cite{romik} Let $\textbf{x}:=(x_i)_{i=1}^n$ be a sequence of distinct real numbers.
Then
\begin{eqnarray*}
  LA_n(\textbf{x}) &=& 1 + \mathbbm{1}(x_1>x_2) + \# \; local \; extremum \; of \; \textbf{x} \\
    &=& 1 + \mathbbm{1}(x_1>x_2) +
\sum_{k=2}^{n-1}\mathbbm{1}(x_{k-1}>x_k<x_{k+1})+\sum_{k=2}^{n-1}\mathbbm{1}(x_{k-1}<x_k>x_{k+1}).
\end{eqnarray*}
\end{lemma}
\begin{example} Let $\textbf{x}=(3,1,7,4,2,6,5).$ Then the local maximums are
$\{x_3,x_6\}=\{7,6\}$ and the local minimums are
$\{x_2,x_5\}=\{1,2\}$. Noting that $x_1
> x_2$ and using Lemma \ref{lakey}, we get $LA_n(\textbf{x}) = 1+1+2+2=6.$ Indeed, the
subsequence  $(3,1,7,2,6,5)$ has length 6 and $\textbf{x}$ does not
have a longer  alternating subsequence.
\end{example}

Now we move on to discussing longest alternating subsequence of a
uniformly random permutation $\pi$. In this direction, \cite{houdre}
and \cite{romik} find the expectation and variance as
$$\mathbb{E}[LA_n(\pi)] = \frac{2n}{3} + \frac{1}{6} \quad \text{and} \quad Var(LA_n(\pi)) = \frac{8n}{45} - \frac{13}{180}.$$
They also prove asymptotic normality of $LA_n(\pi)$ by using an
alternative representation and the underlying local dependence. We
contribute to their result by obtaining convergence rates in the
Kolmogorov distance.

\begin{theorem}
Let $\pi$ be a uniformly random permutation in $S_n.$ Then for every
$n \geq 1$, $$d_K \left(\frac{LA_n(\pi) - \left(\frac{2n}{3} +
\frac{1}{6}\right)}{\sqrt{\frac{8n}{45} - \frac{13}{180}}}, Z\right)
\leq \frac{C}{\sqrt{n}}$$ where $C$ is a constant independent of
$n$.
\end{theorem}

\begin{proof}
Let $\pi$ be a uniformly random permutation and $X_1,...,X_n$
 be independent uniform random variables over
$(0,1)$. Letting
\begin{equation}\label{Eks}
    E_k = \{X_{k-1}>X_k<X_{k+1}\} \cup \{X_{k-1}<X_k>X_{k+1}\} \quad
    \text{for} \quad
 k=2,...,n-1,
\end{equation}
 we have
\begin{eqnarray}\label{tekilk}
  LA_n(\pi) &=& 1 + \mathbbm{1}(\pi(1)>\pi(2)) + \sum_{k=2}^{n-1}\mathbbm{1}(\pi(k-1)>\pi(k)<\pi(k+1)) \nonumber\\
   &+& \sum_{k=2}^{n-1}\mathbbm{1}(\pi(k-1)<\pi(k)>\pi(k+1)) \nonumber\\
    &=_d& 1 + \mathbbm{1}(X_1>X_2) + \sum_{k=2}^{n-1}\mathbbm{1}(X_{k-1}>X_k<X_{k+1}) +
    \sum_{k=2}^{n-1}\mathbbm{1}(X_{k-1}<X_k>X_{k+1}) \nonumber\\
    &=& 1 + \mathbbm{1}(X_1>X_2)+ \sum_{k=2}^{n-1}\mathbbm{1}(E_k)
\end{eqnarray}
where in the second equality, we used the discussion from the
Introduction (or  see  \cite{romik} for a precise statement). Now,
clearly $LA_n(\pi)$ is a sum of 4-dependent random variables and
result follows from Theorem \ref{localdepresult}.
\end{proof}
\begin{remark}
Using the representation of $LA_n(\pi)$ given in (\ref{tekilk}), one
can easily obtain a concentration inequality for  $LA_n(\pi)$ by
using, for example, McDiarmid's well known bounded differences
inequality \cite{Colin1}. By (\ref{tekilk}), we have
$$LA_n(\pi) =_d f(X_1,...,X_n) :=1 + \mathbbm{1}(X_1>X_2)
+\sum_{k=2}^{n-1}\mathbbm{1}(E_k)$$ where $X_i'$s are independent
random variables and $E_k$'s are defined in terms of $X_i$'s as in
(\ref{Eks}). Now by a case analysis, it is easy to see that bounded
differences property holds with $c_k=3$ for $k=1,...,n$ and one
immediately arrives at
\[\mathbb{P}(|LA_n(\pi)- \mu| \geq t) \leq 2e^{-2t^2/9n}.\]
\end{remark}

Next we will work on the same problem for riffle shuffles. First
note that, with its close connection to the number of extremum
points and number of runs, longest alternating subsequences can be
quite useful in non-parametric tests. Indeed, our motivation here
comes from the practical discussions of this issue in \cite{nass} on
cheating in card games.

We start by recalling the development of longest alternating
subsequences in random words given in \cite{houdre}. This time we
need to be careful about defining maxima and minima properly as we
may have repeated values in the sequence. We say that a sequence
$\textbf{x}= (x_1,...,x_n) \in [a]^n$, has a local minimum at $k,$
if (i) $x_k < x_{k+1}$ or $k=n$, and if (ii) for some $j<k$, $x_j >
x_{j+1} = ...=x_{k-1} = x_k$. Similarly, $\textbf{x}$ has a local
maximum at $k$, if $x_k
>x_{k+1}$ or $k=n$, and if (ii) for some $j < k$, $x_j<x_{j+1}
=...=x_{k-1} =x_k$, or for all $j<k$, $x_j = x_k.$ With these
definitions, a useful representation of $LA_n(\textbf{x})$ is found
by Houdre and Restrepo \cite{houdre} as
 $$LA_n(\textbf{x})= \# \; \text{of local maxima of $\textbf{x}$}+\# \; \text{of local minima of $\textbf{x}$}.$$

Letting $\mathbb{X}=(X_1,...,X_n)$ be a random word where $X_i'$s
are independent and uniform over $[a],$ they also show  that
$$\frac{LA_n(\mathbb{X}) - n (2/3-1/3a)}{\sqrt{n} \gamma} \longrightarrow_d
Z$$ as $n \rightarrow \infty$ where

\begin{equation}\label{gamma}
    \gamma^2 = \frac{8}{45} \left(\frac{(1+1/a)(1-3/4a) (1-1/2a)}{1-2/(a+1)}
    \right).
\end{equation}

(Note there is a typo in \cite{houdre} for the expression of
$\gamma^2$. This can be checked from \cite{mansour} by taking limits
in the corresponding variance formula). Now, Lemma \ref{import}, the
discussion just before it with Houdre and Restrepo's result
immediately gives

\begin{theorem}
Let $\rho_{n,a}$ be a random permutation with distribution
$P_{n,a}$. Then
$$\frac{LA_n(\rho_{n,a}) - n (2/3-1/3a)}{\sqrt{n} \gamma}
\longrightarrow_d N(0,1)$$ as $n \rightarrow \infty$ where $\gamma$
is as defined in (\ref{gamma}).
\end{theorem}

This result can be generalized to biased shuffles as in previous
problems in a straightforward way. Asymptotic mean and variance of
this case are described in detail in \cite{houdre}. Also  note that,
due to the lack of local dependence, obtaining convergence rates is
not as easy as the case of uniform random permutations for $a$
shuffles and it will be studied  in a subsequent work. However, when
one focuses on $\rho_{n,2}$, one still has local dependence as we
describe in the rest of this section.

The ease of the 2-shuffle case comes from the following proposition
which gives a characterization of extremum points of 2-shuffles in
terms of the descents. Note that this result also gives the
asymptotic behavior of the number of local maxima or minima with the
use of Theorem \ref{descents}.

\begin{proposition}\label{charext} Let $\rho_{n,2,\textbf{p}}$ be a random permutation with distribution $P_{n,2,\textbf{p}}$ generated by inverse shuffling with
the random vector $\mathbb{X}=(X_1,...,X_n)$ where $X_i$'s are
independent with distribution $\textbf{p}=(p_1,p_2)$ with $0<p_1<1$.
Then for $k=2,...,n-1$,
\begin{itemize}
  \item[i.] $\rho_{n,2,\textbf{p}}$ has a local maximum at $k$ if and only if $\rho_{n,2,\textbf{p}}$ has a
descent at $k$.
  \item[ii.] $\rho_{n,2,\textbf{p}}$ has a local minimum at $k$ if and only if $\rho_{n,2,\textbf{p}}$ has a
descent at $k-1.$
\end{itemize}
\end{proposition}

\begin{proof}
\begin{itemize}
  \item[i.] ($\Rightarrow$) Obvious.  ($\Leftarrow$) Assume $\pi(k)>\pi(k+1)$. We should show
  $\pi(k-1)<\pi(k).$ Since $\pi(k)>\pi(k+1)$, we see that $k^{th}$
  card comes from the second pile and ${k+1}^{st}$ from the first
  pile. Now whether card $k-1$ comes from the first pile or the
  second pile, we have $\pi(k-1)<\pi(k)$ since the relative orders
  of the piles are preserved.
  \item[ii.] Proof is similar to the maximum case and we skip it.
\end{itemize}
\end{proof}

Now we are ready to give a  useful representation of
$LA_n(\rho_{n,2,\textbf{p}}).$ First recall that
\begin{eqnarray*}\label{laperm}
         LA_n(\pi) &=& 1 + \mathbbm{1}(\rho_{n,2,\textbf{p}}(1)>\rho_{n,2,\textbf{p}}(2)) + \sum_{k=2}^{n-1}\mathbbm{1}(\rho_{n,2,\textbf{p}}(k-1)>\rho_{n,2,\textbf{p}}(k)<\rho_{n,2,\textbf{p}}(k+1)) \\
          &+& \sum_{k=2}^{n-1}\mathbbm{1}(\rho_{n,2,\textbf{p}}(k-1)<\rho_{n,2,\textbf{p}}(k)>\rho_{n,2,\textbf{p}}(k+1)).
\end{eqnarray*}

Using Proposition \ref{charext}, we obtain
$$LA_n(\rho_{n,2,\textbf{p}}) =_d 1 +
\mathbbm{1}(X_1>X_2)+\sum_{i=2}^{n-1} \mathbbm{1}(X_i
>X_{i+1}) +\sum_{i=1}^{n-2} \mathbbm{1}(X_i >X_{i+1}) $$ where
$X_i$'s are independent with distribution $\textbf{p}=(p_1,p_2)$.
This immediately gives
\begin{equation}\label{lanrho2}
LA_n(\rho_{n,2,\textbf{p}}) =_d 2 \left(\sum_{i=1}^{n-1}
\mathbbm{1}(X_i
> X_{i+1})\right) +\mathbbm{1}(X_{n-1}<X_n).
\end{equation}

 By the representation in (\ref{lanrho2}), it is clear that we still have
  local dependence for $LA_n(\rho_{n,2,\textbf{p}})$ and thus, we can still use Theorem \ref{localdepresult} with $p=3$ to obtain
a   convergence rate of order $1/\sqrt{n}$ for
$LA_n(\rho_{n,2,\textbf{p}}).$

\section{Concluding Remarks}

In this note, after relating riffle shuffle statistics to random
word statistics, we were able to obtain asymptotic normality results
with convergence rates for the number of descents and inversions
after an arbitrary number of $a$-shuffles. Throughout the way, we
also discussed how similar ideas can be used for a variant of top
$m$ to random shuffles and provided small contributions to Houdre
and Restrepo's work on longest alternating subsequences.

In subsequent work, we will provide convergence rates for the length
of longest alternating subsequences in $a$-shuffles for $a \geq 2.$
We also hope to find out a general framework for establishing the
asymptotic normality of a large class of $a-$shuffle statistics. One
possible direction for this can be using the stochastic dominance
idea introduced in Remark \ref{stocdom} as in many cases it can be
easier to prove the results for 2-shuffles and uniformly random
permutations.


\begin{thebibliography}{1}

\bibitem{diac2} Aldous, D. and Diaconis, P., \emph{Shuffling cards and stopping
times}, American Mathematical Monthly, 93(5) (1986), 333-348.


\bibitem{diac}  Bayer, D. and Diaconis, P., \emph{Trailing the dovetail
shuffle}, Ann. Appl. Probab., 2 (1992) no.2, 294-313.

\bibitem{kousidis} Bliem, T. and Kousidis, S., \emph{The number of flags in finite vector spaces: asymptotic normality and Mahonian statistics}.
J. Algebraic Combin. 37 (2013), no. 2, 361-380.


\bibitem{chenshaolocal} Chen, L. H. Y. and Shao, Q. M., \emph{Normal approximation under
local dependence}, Ann. Prob. 32 (2004), 1985-2028.

\bibitem{chen} Chen, L. H. Y., Shao, Q.,
\emph{Normal approximation for nonlinear statistics using a
concentration inequality approach}, Bernoulli 13 (2007), no. 2,
581-599.

\bibitem{Vis} Conger, M. and Viswanath, D., \emph{Normal approximations for
descents and inversions of permutations of multisets}, J. Theoret.
Prob. (2007) no.2, 309-325.

\bibitem{diac3} Diaconis, P., Fill, K. and Pitman, J., \emph{Analysis of top to random
shuffles}, Combinatorics, Probability and Computing, 1 (1992),
135-155.

\bibitem{fulman}  Fulman, J., \emph{The combinatorics of biased riffle shuffles},
Combinatorica 18 (1998), no.2, 173-184.

\bibitem{gilbert} Gilbert, E., \emph{Theory of shuffling}, Technical
memorandum (1955), Bell Laboratories.


\bibitem{hoef} W. Hoeffding, \emph{A class of statistics with asymptotically normal
distribution}, Ann. Math. Statistics 19 (1948), 293-325.


\bibitem{houdre} Houdre, C. and Restrepo, R., \emph{A probabilistic approach to
the asymptotics of the longest alternating subsequence}, Electron.
J. Combin. 17 (2010), no1.

\bibitem{Janson} Janson, S.,
\emph{Generalized Galois numbers, inversions, lattice paths, Ferrers
diagrams and limit theorems}.Electron. J. Combin. 19 (2012), no. 3,
Paper 34, 16 pp.

\bibitem{mansour} Mansour, T., \emph{Longest alternating subsequences of k-ary words.}
Discrete Appl. Math. 156 (2008), no. 1, 119-124.

\bibitem{Colin1}
McDiarmid, Colin, \emph{Concentration}, Algorithms Combin., 16,
1998.

\bibitem{nass} Nass, C., \emph{Running the Cheaters Out of Town:
Counting Out Corrupt Coins, Dubious Dice, Shifty Shuffling, and
Lying Lotteries }, Unpublished manuscript.

\bibitem{reeds} Reeds, J., \emph{Theory of riffle shuffling},
Unpublished manuscript, 1981.

\bibitem{romik} Romik, D., \emph{Local extrema in random permutations and the
structure of longest alternating subsequences},  23rd International
Conference on Formal Power Series and Algebraic Combinatorics
(2011), 825-834.


\bibitem{stanley} Stanley, R., \emph{Longest alternating subsequences of
permutations}, Michigan Math. J. 57  (2008), 675-687.

\end{thebibliography}
\end{document}